\documentclass[10pt]{article}
\usepackage{amsfonts}
\usepackage{amsmath}
\usepackage{amssymb}
\usepackage{amsthm}
\usepackage{fancyhdr}
\makeatletter
\renewenvironment{proof}[1][\proofname] {\par\pushQED{\qed}\normalfont\topsep6\p@\@plus6\p@\relax\trivlist\item[\hskip\labelsep\bfseries#1\@addpunct{.}]\ignorespaces}{\popQED\endtrivlist\@endpefalse}
\makeatother
\numberwithin{equation}{section}

\usepackage{titlesec}
\titleformat*{\section}{\large\bfseries}

\newtheorem{theorem}{Theorem}[section]
\newtheorem{lemma}{Lemma}[section]

\theoremstyle{definition}
\newtheorem{definition}{Definition}[section]

\begin{document}

\begin{center}
\textbf{On Ranges of Variants of the Divisor Functions that are Dense}  
\vskip 1cm
\footnotesize{
Colin Defant\footnote{Colin Defant \\ 
18434 Hancock Bluff Rd. \\
Dade City, FL 33523}\footnote{This work was supported by National Science Foundation grant no. 1262930.}\\
Department of Mathematics, University of Florida\\ 
1400 Stadium Rd\\ 
Gainesville, FL 32611\\ 
United States\\ 
cdefant@ufl.edu}
\end{center}
\vskip .2 in

\begin{abstract}
\footnotesize{For a real number $t$, let $s_t$ be the multiplicative arithmetic function defined by $\displaystyle{s_t(p^{\alpha})=\sum_{j=0}^{\alpha}(-p^t)^j}$ 
for all primes $p$ and positive integers $\alpha$. We show that the range of a function $s_{-r}$ is dense in the interval $(0,1]$ whenever $r\in(0,1]$. We then find a constant $\eta_A\approx1.9011618$ and show that if $r>1$, then the range of the function $s_{-r}$ is a dense subset of the interval $\displaystyle{\left(\frac{1}{\zeta(r)},1\right]}$ if and only if $r\leq \eta_A$. We end with an open problem. 
\bigskip


} 
\end{abstract}

\section{Introduction} 
Let $\mathbb{N}$ denote the set of positive integers. We will let $p_i$ be the $i^{th}$ prime number, and we will use $\zeta$ to denote the Riemann zeta function.  
\par 
Consider a multiplicative arithmetic function $s_1$ defined by 
\begin{equation} 
s_1(p^{\alpha})=\sum_{j=0}^{\alpha}(-p)^j
\end{equation} 
for all primes $p$ and positive integers $\alpha$. This function, which appears as sequence A061020 in Sloane's Online Encyclopedia of Integer Sequences \cite{OEIS10}, serves as an interesting variant of the well-known sum-of-divisors function $\sigma$. We may generalize the function $s_1$ to a class of functions $s_t$ in the following very natural fashion.  
\begin{definition} \label{Def1.1} 
For any real number $t$, let $s_t$ be the multiplicative arithmetic function defined by \begin{equation}
s_t(p^{\alpha})=\sum_{j=0}^{\alpha}(-p^t)^j.
\end{equation} 
for all primes $p$ and positive integers $\alpha$. 
\end{definition}   
In this paper, we will concentrate on functions $s_{-r}$ for $r>0$, so we will always use $r$ to denote a positive real number. Notice that, for any prime $p$ and nonnegative integer $\alpha$, we have $1-p^{-r}\leq s_{-r}(p^{\alpha})\leq 1$ because  $\displaystyle{\sum_{j=0}^{\alpha}(-p^{-r})^j}$ is an alternating series whose terms have strictly decreasing absolute values. Therefore, if $r>1$ and $N$ is a positive integer with canonical prime factorization $\displaystyle{N=\prod_{j=1}^vq_j^{\beta_j}}$, then we have 
\begin{equation} 
s_{-r}(N)=\prod_{j=1}^vs_{-r}(q_j^{\beta_j})\geq\prod_{j=1}^v(1-q_j^{-r})>\prod_{j=1}^{\infty}(1-p_j^{-r})=\frac{1}{\zeta(r)}. 
\end{equation} 
Hence, for $r>1$, the range of $s_{-r}$ is a subset of the interval $((\zeta(r))^{-1},1]$. We will soon show that, for $r\in(0,1]$, the range of $s_{-r}$ is a dense subset of $(0,1]$. However, we will find that the range of $s_{-2}$ is not dense in $((\zeta(2))^{-1},1]$. Our goal is to find a constant, which we will call $\eta_A$, such that if $r>1$, then the range of $s_{-r}$ is dense in $((\zeta(r))^{-1},1]$ if and only if $r\leq\eta_A$. 
\section{Finding $\eta_A$}  
For the sake of convenience, we introduce a class of functions $L_{-r}$, which we define, for each $r>0$, by $L_{-r}(n)=-\log(s_{-r}(n))$ for all $n\in\mathbb{N}$. Note that the functions $L_{-r}$ take nonnegative values. Furthermore, for any prime $p$, we see that $(L_{-r}(p^{2\alpha+1}))_{\alpha=0}^{\infty}$ forms a decreasing sequence, $(L_{-r}(p^{2\alpha}))_{\alpha=0}^{\infty}$ forms an increasing sequence, and $\displaystyle{\lim_{\alpha\rightarrow\infty}L_{-r}(p^{\alpha})}$ exists (because $\displaystyle{\lim_{\alpha\rightarrow\infty}s_{-r}(p^{\alpha})}$ exists by the Alternating Series test). This motivates us to define an ordering $\succ$ on the nonnegative integers as follows. If $k_1$ and $k_2$ are odd positive integers with $k_1<k_2$, then $k_1\succ k_2$. If $k_1$ and $k_2$ are even nonnegative integers with $k_1<k_2$, then $k_2\succ k_1$. If $k_1$  is an odd positive integer and $k_2$ is an even nonnegative integer, then $k_1\succ k_2$. This ordering has the property that if $r>0$ and $p$ is a prime, then, for any distinct nonnegative integers $k_1$ and $k_2$, $L_{-r}(p^{k_1})>L_{-r}(p^{k_2})$ if and only if $k_1\succ k_2$. 
We are now equipped to prove the following theorem. 
\begin{theorem} \label{Thm2.1} 
If $r\in(0,1]$, then the range of $s_{-r}$ is a dense subset of $(0,1]$. 
\end{theorem}
\begin{proof} 
We first observe that the range of $s_{-r}$ is dense in $(0,1]$ if and only if the range of $L_{-r}$ is dense in $[0,\infty)$. To show that the range of $L_{-r}$ is dense in $[0,\infty)$, we consider the subsums of the series $\displaystyle{\sum_{i=1}^{\infty}L_{-r}(p_i)}$. We see that any finite subsum of this series, say $\displaystyle{\sum_{j=1}^vL_{-r}(q_j)}$, is within the range of $L_{-r}$ because 
\begin{equation}
\sum_{j=1}^vL_{-r}(q_j)=-\log\left(\prod_{j=1}^vs_{-r}(q_j)\right)=L_{-r}\left(\prod_{j=1}^vq_j\right). 
\end{equation} 
Hence, it suffices to show that $\displaystyle{\sum_{i=1}^{\infty}L_{-r}(p_i)}$ is a divergent series whose terms tend to $0$. First, $\displaystyle{\lim_{i\rightarrow\infty}L_{-r}(p_i)=\lim_{i\rightarrow\infty}(-\log(1-p_i^{-r}))=0}$. Second, we know that $\displaystyle{\sum_{i=1}^{\infty}L_{-r}(p_i)}$ diverges because, for $r\in(0,1]$, we have $\displaystyle{\prod_{i=1}^{\infty}(1-p_i^{-r})=0}$. 
\end{proof} 
Henceforth, we will focus on values of $r$ that are greater than $1$. We seek to establish a necessary and sufficient condition for the range of a function $s_{-r}$ to be dense in $((\zeta(r))^{-1},1]$. First, however, we need two lemmata. 
\begin{lemma} \label{Lem2.1} 
If $r>1$, $m\in\mathbb{N}$, and $w\in\{1,2,\ldots,m\}$, then 
\begin{equation} 
1-p_w^{-r}+p_w^{-2r}\leq 1-p_m^{-r}+p_m^{-2r}.
\end{equation}   
\end{lemma} 
\begin{proof} 
Fix some $r>1$. Define $h\colon\mathbb{R}\rightarrow\mathbb{R}$ by $h(x)=1-x^{-r}+x^{-2r}$. Then $h'(x)=rx^{-r-1}(1-2x^{-r})$. If $x\geq 2$, then $h'(x)>0$. As $2\leq p_w\leq p_m$, the result follows. 
\end{proof} 
\begin{lemma} \label{Lem2.2}
Let $r>1$ be a real number, and let $p$ be a prime. For any positive integer $k$, we have $\vert L_{-r}(p^{k+2})-L_{-r}(p^k)\vert<L_{-r}(p^2)$.  
\end{lemma} 
\begin{proof} 
For simplicity, we will write $y=p^{-r}$. First, suppose $k$ is odd. Then, because $(L_{-r}(p^{2\alpha+1}))_{\alpha=0}^{\infty}$ is a decreasing sequence, we have 
\[\vert L_{-r}(p^{k+2})-L_{-r}(p^k)\vert=L_{-r}(p^k)-L_{-r}(p^{k+2})\] 
\[=\log\left(\frac{1}{\sum_{j=0}^k (-y)^j}\right)-\log\left(\frac{1}{\sum_{j=0}^{k+2} (-y)^j}\right)\] 
\begin{equation}
=\log\left(\frac{1+y}{1-y^{k+1}}\right)-\log\left(\frac{1+y}{1-y^{k+3}}\right)=\log\left(\frac{1-y^{k+3}}{1-y^{k+1}}\right). 
\end{equation}
Because $\displaystyle{L_{-r}(p^2)=\log\left(\frac{1}{1-y+y^2}\right)}$, we see that we simply need to show that $\displaystyle{\frac{1-y^{k+3}}{1-y^{k+1}}<\frac{1}{1-y+y^2}}$. 
\par 
Noting that $0<y<\frac{1}{2}$, we have $y^k<y$ and $y^{k+3}<y^{k+2}$. Therefore, $y+y^k+y^{k+3}<2y+y^{k+2}+y^{k+4}<1+y^{k+2}+y^{k+4}$, so we have $y^2+y^{k+1}+y^{k+4}<y+y^{k+3}+y^{k+5}$. After adding $1$ to each side and rearranging terms, we get $1-y+y^2-y^{k+3}+y^{k+4}-y^{k+5}<1-y^{k+1}$, which we may write as $(1-y+y^2)(1-y^{k+3})<1-y^{k+1}$. Hence, $\displaystyle{\frac{1-y^{k+3}}{1-y^{k+1}}<\frac{1}{1-y+y^2}}$, so we have completed the proof for the case in which $k$ is odd. 
\par 
Now, suppose that $k$ is even. Then, because $(L_{-r}(p^{2\alpha}))_{\alpha=0}^{\infty}$ is an increasing sequence, we have 
\[\vert L_{-r}(p^{k+2})-L_{-r}(p^k)\vert=L_{-r}(p^{k+2})-L_{-r}(p^k)\] 
\[=\log\left(\frac{1}{\sum_{j=0}^{k+2} (-y)^j}\right)-\log\left(\frac{1}{\sum_{j=0}^k (-y)^j}\right)\] 
\begin{equation}
=\log\left(\frac{1+y}{1+y^{k+3}}\right)-\log\left(\frac{1+y}{1+y^{k+1}}\right)=\log\left(\frac{1+y^{k+1}}{1+y^{k+3}}\right). 
\end{equation}
Again, we have $\displaystyle{L_{-r}(p^2)=\log\left(\frac{1}{1-y+y^2}\right)}$, so it suffices to show that \\
$\displaystyle{\frac{1+y^{k+1}}{1+y^{k+3}}<\frac{1}{1-y+y^2}}$.
Because $0<y<\frac{1}{2}$, we have $1-y^{2(k+1)}<1-y^{2(k+3)}$. Therefore, $\displaystyle{\frac{1+y^{k+1}}{1+y^{k+3}}<\frac{1-y^{k+3}}{1-y^{k+1}}}$, and we have already shown that $\displaystyle{\frac{1-y^{k+3}}{1-y^{k+1}}<\frac{1}{1-y+y^2}}$. 
\end{proof} 
\begin{theorem} \label{Thm2.2} 
If $r>1$, then the range of $s_{-r}$ is dense in the interval \\ 
$((\zeta(r))^{-1},1]$ if and only if $\displaystyle{s_{-r}(p_m^2)\geq\prod_{i=m+1}^{\infty}s_{-r}(p_i)}$ for all positive integers $m$. 
\end{theorem}   
\begin{proof} 
First, suppose there exists some positive integer $m$ such that \\
$\displaystyle{s_{-r}(p_m^2)<\prod_{i=m+1}^{\infty}s_{-r}(p_i)}$. Let $N$ be an arbitrary positive integer with canonical prime factorization $\displaystyle{N=\prod_{j=1}^v q_j^{\beta_j}}$. If $p_w\vert N$ for some $w\in\{1,2,\ldots,m\}$, then $s_{-r}(N)\leq 1-p_w^{-r}+p_w^{-2r}$. By Lemma \ref{Lem2.1}, we see that $s_{-r}(N)\leq 1-p_m^{-r}+p_m^{-2r}=s_{-r}(p_m^2)$. On the other hand, if $p_w\nmid N$ for all $w\in\{1,2,\ldots,m\}$, then 
\begin{equation}
s_{-r}(N)=s_{-r}\left(\prod_{j=1}^v q_j^{\beta_j}\right)=\prod_{j=1}^v s_{-r}(q_j^{\beta_j})\geq\prod_{j=1}^v s_{-r}(q_j)>\prod_{i=m+1}^{\infty} s_{-r}(p_i). 
\end{equation} 
This shows that there is no element of the range of $s_{-r}$ in the interval \\ 
$\displaystyle{\left(s_{-r}(p_m^2),\prod_{i=m+1}^{\infty}s_{-r}(p_i)\right)}$, so the range of $s_{-r}$ is not dense in $((\zeta(r))^{-1},1]$. 
\par 
To prove the converse, let us suppose that $\displaystyle{s_{-r}(p_m^2)\geq\prod_{i=m+1}^{\infty}s_{-r}(p_i)}$ for all positive integers $m$. We will show that the range of $L_{-r}$ is dense in $[0,\log(\zeta(r)))$, which will prove that the range of $s_{-r}$ is dense in $((\zeta(r))^{-1},1]$. Choose some arbitrary $x\in(0,\log(\zeta(r)))$. We will construct a sequence $(C_n)_{n=1}^{\infty}$ of elements of the range of $L_{-r}$ such that $\displaystyle{\lim_{n\rightarrow\infty}C_n}=x$. First, define $C_0=0$. Now, recall the ordering $\succ$ that we defined at the beginning of this section. We will say that a nonnegative integer $k_1$ is larger than a nonnegative integer $k_2$ with respect to the ordering $\succ$ if and only if $k_1\succ k_2$. Let $n$ be a positive integer. We will ensure by construction that $C_{n-1}\leq x$. If $\displaystyle{C_{n-1}+\lim_{k\rightarrow\infty}L_{-r}(p_n^k)=x}$, then we will define $\alpha_n=-1$. If $\displaystyle{C_{n-1}+\lim_{k\rightarrow\infty}L_{-r}(p_n^k)\neq x}$, then we will define $\alpha_n$ to be the nonnegative integer satisfying $C_{n-1}+L_{-r}(p_n^{\alpha_n})\leq x$ that is largest with respect to the ordering $\succ$. In this case, we define $C_n=C_{n-1}+L_{-r}(p_n^{\alpha_n})$. For now, let us assume that $x$ is such that $\displaystyle{C_{n-1}+\lim_{k\rightarrow\infty}L_{-r}(p_n^k)\neq x}$ for all positive integers $n$. In other words, $\alpha_n\geq 0$ and $C_n$ is defined for all positive integers $n$.  \par 
We first show that $C_n$ is in the range of $L_{-r}$ for all positive integers $n$. Indeed, we have 
\begin{equation}
C_n=\sum_{i=1}^n L_{-r}(p_i^{\alpha_i})=L_{-r}\left(\prod_{i=1}^n p_i^{\alpha_i}\right).
\end{equation}
Now, we defined $(C_n)_{n=1}^{\infty}$ to be a monotonic sequence with the property that $C_n\leq x$ for all $n\in\mathbb{N}$, so we may write $\displaystyle{\lim_{n\rightarrow\infty}C_n=\gamma\leq x}$. Suppose, for the sake of finding a contradiction, that $\gamma<x$.  
For each $n\in\mathbb{N}$, we will let $D_n=L_{-r}(p_n)-L_{-r}(p_n^{\alpha_n})$ and $\displaystyle{E_n=\sum_{i=1}^n D_i}$. 
Then $\displaystyle{C_n+E_n=\sum_{i=1}^n L_{-r}(p_n)}$, so $\displaystyle{\lim_{n\rightarrow\infty}(C_n+E_n)=\lim_{n\rightarrow\infty}\left(-\log\left(\prod_{i=1}^n s_{-r}(p_i)\right)\right)=\log(\zeta(r))}$. Therefore, \\ 
$\displaystyle{\lim_{n\rightarrow\infty}E_n=\log(\zeta(r))-\gamma>\log(\zeta(r))-x}$, 
so we may let $m$ be the smallest positive integer such that $E_m>\log(\zeta(r))-x$. If $\alpha_m=1$ and $m>1$, then $D_m=0$, implying that $E_{m-1}=E_m>\log(\zeta(r))-x$, which contradicts the minimality of $m$. On the other hand, if $\alpha_m=1$ and $m=1$, then $E_m=0>\log(\zeta(r))-x$, which is also a contradiction. Hence, $\alpha_m\neq 1$. If $\alpha_m$ is odd, then we will let $A_m=L_{-r}(p_m^{\alpha_m-2})-L_{-r}(p_m^{\alpha_m})$. In this case, we see, by the definitions of $C_m$ and $\alpha_m$ and the fact that $\alpha_{m-2}\succ\alpha_m$, that $A_m+C_m>x$. If, on the other hand, $\alpha_m$ is even, then we may write $A_m=L_{-r}(p_m^{\alpha_m+2})-L_{-r}(p_m^{\alpha_m})$. Again, by the definitions of $C_m$ and $\alpha_m$ and the fact that $\alpha_{m+2}\succ\alpha_m$, we have $A_m+C_m>x$. No matter the parity of $\alpha_m$, we have $x-C_m<A_m$. Using Lemma \ref{Lem2.2}, we see that $A_m\leq L_{-r}(p_m^2)$, so $x-C_m<L_{-r}(p_m^2)=-\log(s_{-r}(p_m^2))$. As we originally assumed that $\displaystyle{s_{-r}(p_m^2)\geq\prod_{i=m+1}^{\infty}s_{-r}(p_i)}$, we have 
\[x-C_m<-\log(s_{-r}(p_m^2))\leq -\log\left(\prod_{i=m+1}^{\infty}s_{-r}(p_i)\right)\] 
\begin{equation}
=\log(\zeta(r))-(C_m+E_m).
\end{equation} 
This implies that $E_m<\log(\zeta(r))-x$, which is our desired contradiction. This completes the proof of the case in which $\alpha_n\geq 0$ for all $n\in\mathbb{N}$. 
\par 
Finally, let us assume that there is some positive integer $n$ such that \\ 
$\displaystyle{C_{n-1}+\lim_{k\rightarrow\infty}L_{-r}(p_n^k)=x}$. In this case, simply let $C_{n-1+j}=C_{n-1}+L_{-r}(p_n^j)$ for all positive integers $j$. Then, as before, we see that $C_{n-1+j}$ is always in the range of $L_{-r}$. Furthermore, $\displaystyle{\lim_{j\rightarrow\infty}C_{n-1+j}=C_{n-1}+\lim_{j\rightarrow\infty}L_{-r}(p_n^j)=x}$. This completes the proof.  
\end{proof} 
We now have a way to test whether or not the range of $s_{-r}$ is dense in $((\zeta(r))^{-1},1]$ for a given $r>1$. However, after a short lemma, we will be able to simplify the problem even further. 
\begin{lemma} \label{Lem2.3}
If $j\in\mathbb{N}\backslash\{1,2,4\}$, then $\displaystyle{\frac{p_{j+1}}{p_j}<\sqrt{2}}$. 
\end{lemma}
\begin{proof}
A simple manipulation of the corollary to Theorem 3 in \cite{Rosser61} shows that $\displaystyle{\frac{p_{j+1}}{p_j}<\frac{(j+1)(\log(j+1)+\log\log(j+1))}{j\log j}}$ for all integers $j\geq 6$. It is easy to verify that $\displaystyle{\frac{(j+1)(\log(j+1)+\log\log(j+1))}{j\log j}}<\sqrt{2}$ for all $j\geq 32$. Therefore, the desired result holds for $j\geq 32$. A quick search through the values of $\displaystyle{\frac{p_{j+1}}{p_j}}$ for $j<32$ yields the desired result.   
\end{proof} 
\begin{theorem} \label{Thm2.3} 
If $1<r\leq 2$, then the range of $s_{-r}$ is dense in the interval $((\zeta(r))^{-1},1]$ if and only if $\displaystyle{s_{-r}(p_m^2)\geq\prod_{i=m+1}^{\infty}s_{-r}(p_i)}$ for all $m\in\{1,2,4\}$. 
\end{theorem} 
\begin{proof} 
Let us define a function $F$ by $\displaystyle{F(m,r)=s_{-r}(p_m^2)\prod_{i=1}^m s_{-r}(p_i)}$ so that the inequality $\displaystyle{s_{-r}(p_m^2)\geq\prod_{i=m+1}^{\infty}s_{-r}(p_i)}$ is equivalent to $F(m,r)\geq(\zeta(r))^{-1}$.  
Due to the validity of Theorem \ref{Thm2.2}, we see that, in order to prove the result, it suffices to show that if $F(m,r)\geq(\zeta(r))^{-1}$ for all $m\in\{1,2,4\}$, then $F(m,r)\geq(\zeta(r))^{-1}$ for all $m\in\mathbb{N}$. Therefore, let us assume that $r\in(1,2]$ is such that
$F(m,r)\geq(\zeta(r))^{-1}$ for all $m\in\{1,2,4\}$. 
\par 
If $m\in\mathbb{N}\backslash\{1,2,4\}$, then Lemma \ref{Lem2.3} tells us that $p_{m+1}<\sqrt{2}p_m\leq\sqrt[r]{2}p_m$, which means that we may write $2p_{m+1}^{-r}>p_m^{-r}$. As $p_m^{-r}-1$ is negative, we have $2p_{m+1}^{-r}(p_m^{-r}-1)<p_m^{-r}(p_m^{-r}-1)$, so we may write 
$-2p_{m+1}^{-r}+2p_{m+1}^{-2r}=2p_{m+1}^{-r}(p_{m+1}^{-r}-1)<2p_{m+1}^{-r}(p_m^{-r}-1)<p_m^{-r}(p_m^{-r}-1)=-p_m^{-r}+p_m^{-2r}$. Therefore, 
\[F(m+1,r)=s_{-r}(p_{m+1}^2)s_{-r}(p_{m+1})\prod_{i=1}^m s_{-r}(p_i)\] 
\[=(1-p_{m+1}^{-r}+p_{m+1}^{-2r})(1-p_{m+1}^{-r})\prod_{i=1}^m s_{-r}(p_i)\] 
\[=(1-2p_{m+1}^{-r}+2p_{m+1}^{-2r}-p_{m+1}^{-3r})\prod_{i=1}^m s_{-r}(p_i)<(1-2p_{m+1}^{-r}+2p_{m+1}^{-2r})\prod_{i=1}^m s_{-r}(p_i)\] 
\begin{equation}
<(1-p_m^{-r}+p_m^{-2r})\prod_{i=1}^m s_{-r}(p_i)=s_{-r}(p_m^2)\prod_{i=1}^m s_{-r}(p_i)=F(m,r). 
\end{equation}
Thus, if $m\in\mathbb{N}\backslash\{1,2,4\}$, then $F(m+1,r)<F(m,r)$. This means that $F(3,r)>F(4,r)\geq(\zeta(r))^{-1}$. Furthermore, $F(m,r)>(\zeta(r))^{-1}$ for all integers $m\geq 5$ because $(F(m,r))_{m=5}^{\infty}$ is a decreasing sequence and 
$\displaystyle{\lim_{m\rightarrow\infty}F(m,r)=(\zeta(r))^{-1}}$. 
\end{proof} 
Using Mathematica 9.0, we may plot the graphs of $(\zeta(r))^{-1}$, $F(1,r)$, $F(2,r)$, and $F(4,r)$. Doing so, we find that the graphs of $F(2,r)$ and $(\zeta(r))^{-1}$ intersect at a point $r_0\approx1.9011618$. Furthermore, we see 
that if $r\in(1,r_0]$, then $F(m,r)\geq(\zeta(r))^{-1}$ for all $m\in\{1,2,4\}$. Therefore, if $r\in(1,r_0]$, then 
Theorem \ref{Thm2.3} tells us that the range of $s_{-r}$ is dense in the interval $((\zeta(r))^{-1},1]$. One may also verify 
that $F(2,r)<(\zeta(r))^{-1}$ for all $r\in(r_0,3.2)$, so the range of $s_{-r}$ is not dense in $((\zeta(r))^{-1},1]$ whenever $r\in(r_0,3.2)$. This leads us to our final theorem. 
\begin{theorem} 
Let $\eta_A$ be the unique number in the interval $(1,2)$ that satisfies the equation 
\begin{equation}
(1-2^{-\eta_A})(1-3^{-\eta_A})(1-3^{-\eta_A}+3^{-2\eta_A})=\frac{1}{\zeta(\eta_A)}. 
\end{equation} 
If $r>1$, then the range of the function $s_{-r}$ is dense in the interval $\displaystyle{\left(\frac{1}{\zeta(r)},1\right]}$ if and only if $r\leq \eta_A$. 
\end{theorem} 
\begin{proof} 
It is easy to see that the number $\eta_A$ is simply the number $r_0$ discussed in the preceding paragraph. Therefore, in order to prove the theorem, it suffices (in virtue of the preceding paragraph) to show that $\displaystyle{F(1,r)<\frac{1}{\zeta(r)}}$ for all $r\geq 3.2$. For $r\geq 3.2$, we have $\displaystyle{2^{1-2r}+\frac{2}{r-1}<1}$, so \begin{equation} \label{Eq2.10}
2^{1-2r}+\frac{2}{r-1}-\frac{2^{2-r}}{r-1}+\frac{2^{2-2r}}{r-1}-1+2^r<2^r. 
\end{equation}
We may rearrange the left-hand-side of \eqref{Eq2.10} to get 
\begin{equation}
\left(1+2^r+\frac{2}{r-1}\right)\left(1-2^{1-r}+2^{1-2r}\right)<2^r, 
\end{equation} 
from which we obtain 
\begin{equation}
\left(1+2^r+\frac{2}{r-1}\right)\left(1-2^{1-r}+2^{1-2r}-2^{-3r}\right)<2^r. 
\end{equation}    
Therefore, we have 
\[F(1,r)=(1-2^{-r})(1-2^{-r}+2^{-2r})=1-2^{1-r}+2^{1-2r}-2^{-3r}<\frac{2^r}{1+2^r+\frac{2}{r-1}}\] 
\begin{equation}
=\left(1+\frac{1}{2^r}+\frac{1}{2^{r-1}(r-1)}\right)^{-1}=\left(1+\frac{1}{2^r}+\int_2^{\infty}\frac{1}{x^r}dx\right)^{-1}<\frac{1}{\zeta(r)}. 
\end{equation}
\end{proof} 
\section{An Open Problem} 
In this paper, we have found necessary and sufficient conditions for the range of a function $s_{-r}$ to be dense in $((\zeta(r))^{-1},1]$ (for $r>1$). In other words, we know exactly when the closure of the range of a function $s_{-r}$ will be the interval $[(\zeta(r))^{-1},1]$. This point of view prompts the following more general question. If we are given a positive integer $L$, then what are the values of $r>1$ such that the closure of the range of the function $s_{-r}$ is a disjoint union of exactly $L$ subintervals of $[(\zeta(r))^{-1},1]$? 

\section{Acknowledgments} 
Dedicated to Miss Raleigh S. Howard. 
\par 
The author would like to thank the unknown referee for his or her helpful advice. The author would also like to thank Professor Peter Johnson for inviting him to the 2014 REU in Algebra and Discrete Mathematics.

\end{document}